\documentclass{article}
\usepackage{latexsym,eepic,makeidx,fancybox,amstext,amssymb,amsmath,amsthm,youngtab}
\usepackage[dvips]{graphicx} \usepackage[latin1]{inputenc} \usepackage[english]{babel}
\newtheorem{theorem}{Theorem}
\newtheorem{lemma}{Lemma}

\newcommand{\p}{^\prime}

\newcommand{\mbf}{\mathbf}

\title{Stationary probability of the identity for the TASEP on a ring}
\author{Erik Aas}\date{}
\begin{document} 
\maketitle
\abstract{Consider the following Markov chain on permutations of length $n$. At each time step we choose a random position. If the letter at that position is smaller than the letter immediately to the left (cyclically) then these letters swap positions. Otherwise nothing happens, corresponding to a loop in the Markov chain. This is the circular TASEP. We compute the average proportion of time the chain spends at the identity permutation (and, in greater generality, at sorted words). This answers a conjecture by Thomas Lam \cite{Lam}.}

\section{The totally asymmetric simple exclusion process}

Let $w$ be a finite word on the alphabet $\{1,2,\dots\}$.
By the {\it type} of $w$ we mean the vector $\mbf{m} = \mbf{m}(u) = (m_1,\dots,m_r)$, where $m_i$ is the number of occurrences of the letter $i$ in $w$.
We will consider words with the property that $m_1, m_2, \dots, m_r > 0$, $m_{r+1} = m_{r+2} = \dots = 0$ for some $r > 0$.

Given a vector $\mbf{m} = (m_1,\dots,m_r)$ of positive integers, we define a Markov chain on all words of type $\mbf{m}$, called the $\mbf{m}$-totally asymmetric exclusion process, or $\mbf{m}$-TASEP.
Let $n = m_1+\dots+m_r$.
To describe the transitions in this chain, let $u$ be an arbitrary word of type $\mbf{m}$ (ie. an arbitrary node in the chain).

For each position $i$ in $u$, if the letter at that position is strictly smaller than that immediately to the left (cyclically), then there is a transition from $u$ to the word where these two letters are swapped. The probability of this transition is $\frac{1}{n}$.

Each node $u$ has a loop, which is assigned a probability such that the sum of probabilities of outgoing transitions is $1$.

This defines a Markov chain with a unique stationary distribution, see \cite{MLQ}, which we denote by $\pi$.
More explicitly, this means that 
$k \pi(u) = \sum_{v\to u} \pi(v)$, where $k$ is the number of transitions going from $u$ to some other node, and the sum is over all $v \neq u$ having a transition to $u$.

In his study of reduced expressions of elements of affine Weyl groups in terms of simple generators, Lam \cite{Lam} defined a natural Markov chain on the corresponding hyperplane arrangements. Further, he defined another natural Markov chain on the corresponding finite Weyl group (which is in a certain precise sense a projection of the former one), which for the type A Weyl groups corresponds to the TASEP where there is only one particle of each class, i.e. the nonzero $m_i$'s are all equal to $1$. In this particular case, some properties of the chain can be translated to properties of a large random $N$-core (see \cite{Lam} for details).

The original motivation for the TASEP, however, comes from physics. See \cite{Liggett} for an extensive survey.

In Section \ref{mlqsec} we describe {\it multi-line queues}, introduced by Ferrari and Martin to describe the stationary distribution of the TASEP \cite{MLQ}.

In Section \ref{finishsec} we use multi-line queues to evaluate the stationary distribution at {\it sorted} words.

The TASEP has the following important and well-known property.

\begin{lemma}
\label{probid}
Let $h(u)$ be the word obtained from $u$ by replacing all occurrences of $r$ by $r-1$, and $v$ any word.
Then $\sum_u \pi(u) = \pi(v)$ summing over all words $u$ of some fixed type $\mbf{m}$ such that $h(u) = v$.
\end{lemma}
\begin{proof}
	It is easy to check that the map $f$ describes a coupling of the $\mbf{m}$-TASEP and the $\mbf{m}\p$-TASEP, where $\mbf{m}\p$ is the type of $v$.
\end{proof}
\label{mlqsec}

An $\mbf{m}$-multi-line queue, or $\mbf{m}$-MLQ for short, is a rectangular $r \times n$ array, where exactly $m_1+\dots+m_i$ of the entries in row $i$ are marked as {\it boxes}. The other entries are considered empty. These remarkable objects were introduced by Ferrari and Martin to describe the stationary distribution of the TASEP, see \cite{MLQ}.

To any MLQ we associate a labelling of its boxes, as follows. Each box in row $1$ is labelled $1$.

Suppose we have labelled all rows with index less than $i$, and no box in row $i$ is labelled.
We now describe how to label the boxes in row $i$.
First choose any ordering of the boxes in row $i-1$ such that if the label in box $B$ is smaller than that in box $B\p$, then $B$ comes before $B\p$ in the ordering.
Now go through the boxes in row $i-1$ according to this ordering. When examining a box with label $l$, find the first box weakly to the right (cyclically) not already labelled, and label it by $l$.
When all boxes in row $i-1$ have been examined, label the remaining unlabelled boxes in row $i$ (which should be $m_i$ in number) by $i$.
The labelling will not depend on the orderings chosen (as is easy to see).

\section*{Example}

Below is an example of an $(2,1,1,1,3,1,1)$-MLQ.

\vspace{0.5cm}
$
\left|
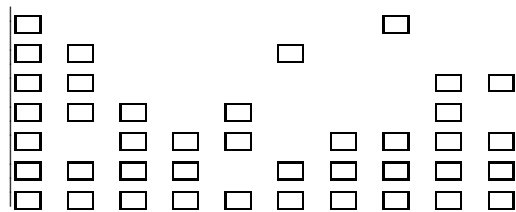
\begin{matrix}
\fbox{ } &          &          &          &          &          &          & \fbox{ } &          &           \\
\fbox{ } & \fbox{ } &          &          &          & \fbox{ } &          &          &          &           \\
\fbox{ } & \fbox{ } &          &          &          &          &          &          & \fbox{ } & \fbox{ }  \\
\fbox{ } & \fbox{ } & \fbox{ } &          & \fbox{ } &          &          &          & \fbox{ } &           \\
\fbox{ } &          & \fbox{ } & \fbox{ } & \fbox{ } &          & \fbox{ } & \fbox{ } & \fbox{ } & \fbox{ }  \\
\fbox{ } & \fbox{ } & \fbox{ } & \fbox{ } &          & \fbox{ } & \fbox{ } & \fbox{ } & \fbox{ } & \fbox{ }  \\
\fbox{ } & \fbox{ } & \fbox{ } & \fbox{ } & \fbox{ } & \fbox{ } & \fbox{ } & \fbox{ } & \fbox{ } & \fbox{ }  \\
\end{matrix}
\right|
$

\vspace{0.5cm}
\noindent The associated labelling is given by
\vspace{0.5cm}

$
\left|
\begin{matrix}
\fbox{1} &          &          &          &          &          &          & \fbox{1} &          &           \\
\fbox{1} & \fbox{1} &          &          &          & \fbox{2} &          &          &          &           \\
\fbox{1} & \fbox{1} &          &          &          &          &          &          & \fbox{2} & \fbox{3}  \\
\fbox{1} & \fbox{1} & \fbox{3} &          & \fbox{4} &          &          &          & \fbox{2} &           \\
\fbox{1} &          & \fbox{1} & \fbox{3} & \fbox{4} &          & \fbox{5} & \fbox{5} & \fbox{2} & \fbox{5}  \\
\fbox{1} & \fbox{6} & \fbox{1} & \fbox{3} &          & \fbox{4} & \fbox{5} & \fbox{5} & \fbox{2} & \fbox{5}  \\
\fbox{1} & \fbox{6} & \fbox{1} & \fbox{3} & \fbox{7} & \fbox{4} & \fbox{5} & \fbox{5} & \fbox{2} & \fbox{5}  \\
\end{matrix}
\right|.
$
\vspace{0.5cm}

The bottom row of a $\mbf{m}$-MLQ thus consists of a contiguous string of boxes, whose labels form a word $u$ of type $\mbf{m}$.
We say that the MLQ {\it represents} $u$. Denote the number of MLQ's representing $u$ by $[u]$.\footnote[1]{This notation collides with the notation $[n] = \{1,\dots,n\}$, though not in any serious way.}
The number of $\mbf{m}$-MLQ's is clearly $Z_{\mbf{m}} := \sum_u [u] = \prod_{i=1}^r {n \choose m_1+\dots+m_i}$.

\begin{theorem}
\label{mlq}
[Ferrari and Martin, \cite{MLQ}]
For any word $u$, $\pi(u) = [u] / Z_{\mbf{m}}$, where $\mbf{m}$ is the type of $u$.
\end{theorem}

Using this theorem, it is not too difficult (compared to, say, solving the linear equations defining the stationary distribution) to work out the values $[u]$ for words of type $(1,1,1,1)$. They are listed below. Note that for any cyclic shift $u\p$ of $u$, we have $[u\p] = [u]$.

\begin{verbatim}
[1234] = 9
[1243] = 3
[1324] = 3
[1342] = 3
[1423] = 5
[1432] = 1
\end{verbatim}

\noindent
The MLQ's corresponding to $[1423] = 5$ are the following ones.

\vspace{0.5cm}
\noindent
$
\left|
\begin{matrix}
\fbox{1} &          &          &          \\
\fbox{1} &          & \fbox{2} &          \\
\fbox{1} &          & \fbox{2} & \fbox{3} \\
\fbox{1} & \fbox{4} & \fbox{2} & \fbox{3} \\
\end{matrix}
\right|
$
$
\left|
\begin{matrix}
         &          &          & \fbox{1} \\
\fbox{1} &          & \fbox{2} &          \\
\fbox{1} &          & \fbox{2} & \fbox{3} \\
\fbox{1} & \fbox{4} & \fbox{2} & \fbox{3} \\
\end{matrix}
\right|
$
$
\left|
\begin{matrix}
\fbox{1} &          &          &          \\
\fbox{1} & \fbox{2} &          &          \\
\fbox{1} &          & \fbox{2} & \fbox{3} \\
\fbox{1} & \fbox{4} & \fbox{2} & \fbox{3} \\
\end{matrix}
\right|
$
\\
\vspace{0.5cm}
\\
$
\left|
\begin{matrix}
         &          &          & \fbox{1} \\
\fbox{1} & \fbox{2} &          &          \\
\fbox{1} &          & \fbox{2} & \fbox{3} \\
\fbox{1} & \fbox{4} & \fbox{2} & \fbox{3} \\
\end{matrix}
\right|
$
$
\left|
\begin{matrix}
         &          & \fbox{1} &          \\
\fbox{1} & \fbox{2} &          &          \\
\fbox{1} &          & \fbox{2} & \fbox{3} \\
\fbox{1} & \fbox{4} & \fbox{2} & \fbox{3} \\
\end{matrix}
\right|
$
\vspace{0.5cm}

One can similarly compute $[{\tt 1233}]=3$. Below, we will generalize the following (arithmetic) identities: $[{\tt 1234}] = {2 + 1 \choose 1}[{\tt 1233}]$, $[{\tt 1234}] + [{\tt 1243}] = {4 \choose 1}[{\tt 1233}]$ and $[{\tt 1243}] = [{\tt 1233}]$.

\section{Sorted words}
\label{finishsec}

In this section we compute the stationary probability of any {\it sorted} word (such as {\tt 1234} or {\tt 112334}).

Let us fix $r\geq 3$, and a word $u$ on $\{1,\dots,r-2\}$ of length $s$.
For notational convenience we will identify the letters $r-1$ and $r$ with $\alpha$ and $\beta$ respectively.

\begin{lemma}
\label{ba_to_aa}
Let $w$ be any word in the alphabet $[r]$. Then $[\beta\alpha w] = [\alpha\alpha w]$.
\end{lemma}
\begin{proof}
Consider any MLQ counting $[\beta\alpha w]$. The last three rows look like this:

$
\left|
\begin{matrix}
X & & \dots \\
Y & \fbox{$\alpha$} & \dots \\
\fbox{$\beta$} & \fbox{$\alpha$} & \dots \\
\end{matrix}
\right|.
$

The positions marked X and Y have to be empty from boxes by the definition of the labelling of the MLQ. Let us make the following simple change in the MLQ:

$
\left|
\begin{matrix}
& & \dots \\
\fbox{$\alpha$} & \fbox{$\alpha$} & \dots \\
\fbox{$\alpha$} & \fbox{$\alpha$} & \dots \\
\end{matrix}
\right|.
$

The operation just described is easily checked to be a bijection between the MLQs counting $[\beta\alpha w]$ and the MLQs counting $[\alpha\alpha w]$.
\end{proof}
 
For any $b \geq 0$, let $E_b$ denote the set of words in $\{\alpha, \beta\}$ of length $n-s$ with $b$ occurrences of $\beta$ (and thus $n-s-b$ occurrences of $\alpha$).

We denote the unique sorted word in $E_{b}$ by $e^{(b)}$. For a word $v \in E_b$, denote by $f(v)$ the word obtained by changing all non-trailing occurrences of $\beta$ into $\alpha$
(example:$f(\beta \alpha\beta \beta \alpha\alpha\beta\beta\beta)
					 =\alpha\alpha\alpha\alpha\alpha\alpha\beta\beta\beta$) and by $g(v)$ the number of occurrences of $\beta$ in $f(v)$ ($3$ in the example).

Thus $f(v) = e^{(g(v))}$. The number of $v\in E_b$ satisfying $g(v) = k$ is clearly ${n-s-k-1\choose b-k}$.

Let us state an immediate consequence of Lemma \ref{ba_to_aa}.

\begin{lemma}
For any $v\in E_b$, we have $[uv] = [uv\p]$, where $v\p=f(v)$.
\end{lemma}

\begin{lemma}
We have $\sum_{v\in E_b} [uv] = [ue^{(0)}] = {n\choose b}[ue^{(0)}]$
\end{lemma}

\begin{proof}
Let $\mbf{m}_1$ be the type of the words $uv$ for $v\in E_b$, and $\mbf{m}_2$ be the type of $uv^{(0)}$.
According to Lemma \ref{probid} we have (replacing $\pi(\cdot)$ by $[\cdot]$ by using Theorem \ref{mlq})
$\frac{1}{Z_{\mbf{m}_1}} \sum_{v\in E_b} [uv] = \frac{1}{Z_{\mbf{m}_2}} [ue^{(0)}]$.
Since $Z_{\mbf{m}_1} / Z_{\mbf{m}_2} = {n\choose b}$, we are done.
\end{proof}

For the proof of Lemma \ref{final}, we will need the following classical identity for the binomial coefficients.
\begin{lemma}
For any $n,b,s\geq 0$, 
${n \choose b} = \sum_{k=0}^b {n-s-k-1\choose b-k}{s+k\choose s}$.
\end{lemma}
\begin{proof}
Let $A_k = \{\sigma\subseteq[n]:|\sigma|=b$, $|[s+k]\cap\sigma| = k$, and $s+k+1\in\sigma\}$.
Then ${[n]\choose b}=\bigcup_{k=0}^b A_k$. The size $|A_k|$ is precisely ${n-s-k-1\choose b-k}{s+k\choose k}$.
\end{proof}

\begin{lemma}
\label{final}
For any $b \geq 0$, $[ue^{(b)}] = {s+b\choose s}[ue^{(0)}]$.
\end{lemma}

\begin{proof}
We induct on $b$, the case $b = 0$ being clear.

By picking out the $k = b$ term from the sum in the right hand side of the following equation,

$$
{n\choose b}[ue^{(0)}] =
\sum_{v\in E_b}[uv] =
\sum_{k=0}^b \sum_{v\in E_b: g(v) = k}[uv] =
$$

$$
\sum_{k=0}^b \sum_{v\in E_b: g(v) = k}[ue^{(k)}] =
\sum_{k=0}^b {n-s-k-1 \choose b-k}[ue^{(k)}],
$$ 

we obtain
$$
[ue^{(b)}] = \sum_{v\in E_b} [uv] - \sum_{k=0}^{b-1}{n-s-k-1 \choose b-k}[ue^{(k)}] =
$$

$$
[ue^{(0)}]\left( {n \choose b} - \sum_{k=0}^{b-1} {n-s-k-1\choose b-k}{s+k\choose s} \right) =
[ue^{(0)}]{s+b\choose b},$$ completing the induction.
\end{proof}

By applying Lemma \ref{final}, $r-2$ times, we obtain the following
\begin{theorem}
\label{finish}
Let $w$ denote the sorted word on $\{1,\dots,r\}$ with $m_i$ occurrences of the letter $i$. (i.e. $w = 1^{m_1}2^{m_2}\dots r^{m_r}$).
Then $$[w] = \prod_{i=2} ^{r-1} {n-m_i \choose m_1+\dots+m_{i-1}}.$$
\end{theorem}

Letting $m_1 = \dots = m_r = 1$ (and thus $r = n$), this answers Conjecture 1 in \cite{Lam} affirmatively.

In fact, the same reasoning used to prove Theorem \ref{finish} can be used to express any bracket $[uv]$, where $v$ is a sorted word all of whose letters are greater than all letters in $u$, in terms of brackets of 'simpler' words.
Using the (well-known) lemma below, a similar remark can be made for brackets $[uv]$ where $v$ is a sorted word all of whose letters are smaller than all letters in $u$.

\begin{lemma}
Let $w$ be a word on the alphabet $[r]$, and $w\p$ the word obtained from $w$ by replacing each letter $i$ by $r+1-i$ and reversing the string. Then $\pi(w)=\pi(w\p)$ and (consequently) $[w] = [w\p]$.
\end{lemma}
\begin{proof}
The TASEP can be alternatively described as "larger letters jump to the right if they swap with a smaller letter". This description is converted to the original one by the operation described.
\end{proof}

An interesting consequence of Theorem \ref{finish} is that the stationary probability of the sorted word of type $\mbf{m} = (m_1, \dots, m_r)$ depends only on the values but not on the order of the $m_i$. We have failed to find a probabilistic explanation of this fact.

The TASEP has recently been extended Lam and Williams (see \cite{LW}, also \cite{AM}, \cite{AL}) to an inhomogenous version where particles of class $i$ jump at rate $x_i$, for arbitrary numbers $x_1, \dots, x_{r-1}$. In a forthcoming paper we will prove corresponding identities for this version.

To describe those formulas, let $v_i = 1/x_i$ for each $i$, and $h_k$ be the homogenous symmetric polynomial of degree $k$

Then (after a multiplying by a suitable monomial factor), the corresponding formula for the sorted word of type $\mbf{m} = (m_1, \dots, m_r)$ is
$$
\prod_{j=1} ^r \left(\sum_{(t_1,\dots,t_j)} \left(\prod_{i=1}^{j-1} {m_i+t_i-1\choose m_i-1} v_i^{t_i}\right)v_j^{t_j} \right).
$$

Similarly (and with the same normalizing factor), the generalization of $Z_{\mbf{m}}$ is 
$$
\prod_{j=1} ^r \left(\sum_{(t_1,\dots,t_j)} \left(\prod_{i=1}^{j-1} {m_i+t_i-1\choose m_i-1} v_i^{t_i}\right){m_j+t_j\choose m_j}v_j^{t_j} \right).
$$

In both formulas, and for each $j$, the sum ranges over all $j$-tuples $(t_1,\dots,t_j)$ of nonnegative integers such that $t_1+\dots+t_j = n - (m_1 + \dots + m_j)$. These generalized formulas reduce to the ones above on letting $v_1 = \dots = v_{n-1} = 1$.
\vspace{0.5cm}

\noindent
{\bf Acknowledgement}

I thank my advisor Svante Linusson for suggesting the problem and for providing many helpful suggestions on an early version of this note.

\end{document}